\documentclass[reqno, 11pt]{amsart}
\usepackage{amsmath,mathtools}
 \usepackage{amssymb}
\usepackage{amsthm}
\usepackage{amsmath}
\usepackage{times}
\usepackage{latexsym}
\usepackage[mathscr]{eucal}
\usepackage{graphicx}
\usepackage{float}

\numberwithin{equation}{section}
 
  \newtheorem{theorem}{Theorem}[section]
  
  \newtheorem{lemma}[theorem]{Lemma}
  \newtheorem{corollary}[theorem]{Corollary}

\title[Curvature symmetries in lightlike hypersurfaces]{On the curvature symmetries in lightlike hypersurfaces of an indefinite Sasakian manifold}
\author[Samuel Ssekajja]{Samuel Ssekajja}
\newcommand{\acr}{\newline\indent}
\address{ School of Mathematics\acr
 University of the Witwatersrand\acr
 Private Bag 3, Wits 2050\acr
South Africa}
\email{samuel.ssekajja@wits.ac.za} 
\thanks{}


\begin{document}

\begin{abstract}
In this note, we show that a lightlike hypersurface of an indefinite Sasakian manifold, which is tangent to structure vector field is not locally symmetric, semi-symmetric or semi-parallel.

\end{abstract}
\maketitle
\section{Introduction} 
The study of curvature symmetries in lightlike hypersurfaces $(M, g)$ of semi-Euclidean spaces $(\bar{M}, \bar{g})$ were introduced and studied by B. Sahin \cite{Sahins}. In that paper, the author defines semi-symmetric hypersurfaces to be those satisfying the condition $R\cdot R=0$, where $R$ is the induced curvature tensor of the hypersurface (see Definition 4). Other symmetries introduced therein are Ricci semi-symmetry, semi-parallel as well as local symmetry, in which the following conditions are assumed $R\cdot \mathrm{Ric}=0$, $\tilde{R}\cdot h=0$ and $\nabla R=0$, respectively. Here, $\mathrm{Ric}$ denotes the induced Ricci tensor of the hypersurface, $\tilde{R}$ the curvature tensor of its Van der Waerden-Bortolotti connection, which is built from its induced connection $\nabla$ and its transversal connection, and $h$ its second fundamental form. The results of the above paper motivated other authors to study the same hypersurfaces in different ambient spaces. For example, see \cite{Sahin, Oscar1, Oscar2, Gupta}. Although it seems natural to extend the study of such hypersurfaces in other spaces such as indefinite almost contact metric manifolds $(\bar{M}, \bar{\phi}, \zeta, \eta, \bar{g})$, like indefinite Sasakain manifolds, the presence of the structure vector field $\zeta$ posses some hidden challenges. In fact, the position of $\zeta$  relative to any hypersurface greatly affects its geometry. For example, we have showed, in this paper, that if a lightlike hypersurface of an indefinite Sasakian manifold is tangent to  $\zeta$, then it can not exhibit any of the above symmetries. These findings can be seen in Theorems \ref{c8}, \ref{e47}, \ref{e471} and \ref{e4711} of this paper. Our results greatly affects a number of prior assumptions in which such hypersurfaces were assumed to exist in indefinite Sasakian space forms, i.e. Sasakian spaces of constant $\bar{\phi}$-holomorphic sectional curvatures, such as in \cite{Oscar1} and \cite{Gupta}. 

The study of lightlike hypersurfaces has attracted the attention of many authors since its introduction by K. L. Duggal and A. Bejancu \cite{Duggal5}. This is due to their numerous applications in mathematical physics and electromagetism. More findings on the study of lightlike hypersurfaces, and submifoolds in general, can be found in \cite{Calin2, Duggal6, Duggal4, Sahin, Jin4, Jin2, Jin3, Ssekajja}, and any other references cited therein. The rest of this paper is arranged as follows: In Sections \ref{pre} and \ref{cons}, we quote some basic information on lightlike hypersurfaces and indefinite Sasakain manifolds used in this paper. Sections \ref{loc}, \ref{semi-sym}, and \ref{semi-par}, we investigate locally symmetric, semi-symmetric and semi-parallel lightlike hypersurfaces of an indefinite Sasakian manifold $(\bar{M}, \bar{\phi}, \zeta, \eta, \bar{g})$, which are tangent to the structure vector field $\zeta$.
\section{Preliminaries} \label{pre}
An odd-dimensional semi-Riemannian manifold $\bar{M}$ is called an indefinite Sasakian manifold \cite{Takahashi, Tanno} if there exists a $(1, 1)$-tensor field $\bar{\phi}$, a vector field $\zeta$, called structure vector field, a 1-form $\eta$ and an indefinite metric tensor $\bar{g}$ on $\bar{M}$, such that
\begin{align}
\bar{\phi}^{2}X&=-X+\eta(X)\zeta,\label{e1}\\
\eta(\xi)&=1,\label{e2}\\
\eta(\bar{\phi}X)&=0,\label{e8}\\
\bar{g}(X,\xi)&=\eta(X),\label{e3}\\
\bar{g}(\bar{\phi}X,\bar{\phi}Y)&=\bar{g}(X,Y)-\eta(X)\eta(Y),\label{e4}\\
d\eta(X,Y)&=\bar{g}(\bar{\phi}X,Y),\label{e5}\\
(\bar{\nabla}_{X}\bar{\phi})Y&=\eta(Y)X-g(X,Y)\zeta,\label{e6}
\end{align}
where $X$ and $Y$ are any tangent vector fields to $\bar{M}$, and $\bar{\nabla}$ is the semi-Riemannian connection of $\bar{g}$. Replacing $Y$ with $\zeta$ in (\ref{e6}), we get 
\begin{align}\label{e7}
	\bar{\nabla}_{X}\zeta=\bar{\phi}X,
\end{align}
for all $X$ tangent to $\bar{M}$. Let $(\bar{M}, \bar{g})$ be a $(m+2)$-dimensional semi-Riemannian manifold of index $q$, where $0<q<m+2$. Let $(M,g)$ be a hypersurface of $\bar{M}$, where $g$ denotes the metric tensor field induced on $M$ by $\bar{g}$. Then, $M$ is called a lightlike hypersurface of $\bar{M}$ whenever $rank(g)=m$. In this case,  the normal vector bundle $TM^{\perp}$ intersects the tangent
bundle $TM$ of $M$ along a nontrivial differentiable distribution known as the radical distribution of $M$, denoted as $\mathrm{Rad}\, TM$, and given by $\mathrm{Rad}\, TM:p\mapsto \mathrm{Rad}\, T_{p}M=T_{p}M\cap T_{p}M^{\perp}$. Furthermore, it is known \cite{Duggal5, Duggal6} that the complementary bundle to $TM^{\perp}$ in $TM$, called the screen distribution, denoted by $S(TM)$, is non-degenerate and the following decomposition holds: 
 \begin{align}\label{n97}
TM=S(TM)\perp \mathrm{Rad}\, TM,
\end{align}
where $\perp$ denotes the orthogonal direct sum. Moreover, a result due to \cite{Duggal5, Duggal6} shows that there exists a unique vector bundle $\mathrm{tr}(TM)$, called the  lightlike transversal bundle of $M$ with respect to $S(TM)$,  of rank 1 over $M$ such that for any non-zero section $\xi$ of $\mathrm{Rad}\, TM$ on a coordinate neighbourhood $\mathcal{U}\subset M$, there exists a unique section $N$ of $\mathrm{tr}(TM)$ on $\mathcal{U}$ satisfying the conditions 
\begin{align}\label{er1}
	\bar{g}(\xi,N)=1, \quad \bar{g}(N,N)=0 \quad \mbox{and}\quad \bar{g}(N,W)=0,
\end{align}
for any $W$ tangent to $S(TM)$. As a consequence, we have the decomposition  
\begin{align}\label{s15}
	T\bar{M}|_{M}&=S(TM)\perp \{\mathrm{Rad}\, TM\oplus \mathrm{tr}(TM)\}=TM\oplus \mathrm{tr}(TM),
\end{align}
where $\oplus$ denotes a direct sum, not necessarily orthogonal. 

Let $\bar{\nabla}$ be the Levi-Civita connection of $\bar{M}$ and let $P$ be the projection morphism of $TM$ onto $S(TM)$, with respect to (\ref{n97}). Then the local Gauss-Weingarten equations of $M$ and $S(TM)$ are given by (see details in  \cite{Duggal5, Duggal6}):
\begin{align}
 \bar{\nabla}_{X}Y&=\nabla_{X}Y+h(X,Y)=\nabla_{X}Y+B(X,Y)N,\label{s10}\\
  \bar{\nabla}_{X}N&=-A_{N}X+\nabla^{t}_{X}N=-A_{N}X+\tau(X)N, \label{s11}\\
  \nabla_{X}PY&=\nabla^{*}_{X}PY+h^{*t} (X,PY)=\nabla^{*}_{X}PY + C(X,PY)\xi,\label{s12}\\
  \nabla_{X}\xi&=-A^{*}_{\xi}X+\nabla^{*t}_{X}\xi=-A^{*}_{\xi}X -\tau(X) \xi,\label{s13}
 \end{align}
for all $X$ and $Y$ tangent to $M$, $\xi$ tangent to $\mathrm{Rad}\, TM$ and $N$ tangent to $\mathrm{tr}(T M)$.  In the above, $\nabla$ and $\nabla^{*}$ are the induced linear connections on $TM$ and $S(TM)$, respectively. Furthermore, $h$ and $B$ denotes the second fundamental form and local second fundamental form of $M$. $h^{*t}$ and $C$ are the second fundamental form and local second funadamental form of $S(TM)$.  $A_{N}$ and $A^{*}_{\xi}$ are the shape operators of $TM$ and $S(TM)$, respectively. $\nabla^{t}$ and $\nabla^{*t}$ are linear connections on $\mathrm{tr}(TM)$ and $\mathrm{Rad}\, TM$, respectively, while $\tau$ is a smooth 1-form on $TM$. We note that both shape operators $A^{*}_{\xi}$ and $A_{N}$ are screen-valued, and satisfy
\begin{align}\label{er7}
g(A^{*}_{\xi}X,Y)=B(X,Y), \quad g(A_{N}X,PY)= C(X,PY),
\end{align}
for any $X$ and $Y$ tangent to $M$. The induced connection $\nabla^{*}$ is a metric connection, while $\nabla$ is generally not, as it satisfies the relation
\begin{align}\label{s14}
	(\nabla_{X}g)(Y,Z)=B(X,Y)\theta(Z)+B(X,Z)\theta(Y), 
\end{align}
for all $X$, $Y$ and $Z$ tangent to $M$, where $\theta$ is a smooth 1-form $M$, given by 
\begin{align}\label{er6}
\theta(X)=\bar{g}(X,N),
\end{align}
for any $X$ tangent to $M$. Due to the first relation in (\ref{er1}), we notice that the 1-form $\theta$ is nowhere identically vanishing on $M$. In fact, we have 
\begin{align}\label{er2}
\theta(\xi)=\bar{g}(\xi, N)=1.
\end{align}

\section{Basic constructions on lightlike hypersurfaces}\label{cons}
 Consinder a lightlike hypersurface $(M, g)$ of an indefinite Sasakian manifold $(\bar{M}, \bar{\phi},\zeta, \eta, \bar{g})$, which is tangent to the structure vector field $\zeta$. C. Calin \cite{Calin2} has proved that $\zeta$ belongs to the screen distribution $S(TM)$. In this paper, we assume that $M$ is tangent to $\zeta$. From the fact $\bar{g}(\bar{\phi}\xi,\xi)=0$, we notice that the lightlike vector field $\bar{\phi}\xi$ is tangent to $M$. Thus, we consider a screen distribution $S(TM)$ on $M$ containing $\bar{\phi}\mathrm{Rad}\, TM$ as a vector subbundle. Consequently, $N$ is orthogonal to $\bar{\phi}\xi$ and we have $\bar{g}(\bar{\phi}N,\xi)=-\bar{g}(N,\bar{\phi}\xi)=0$. Since $\bar{g}(\bar{\phi}N,N)=0$, it follows that $\bar{\phi}N$ is tangent to $M$ and in particular, it belongs to $S(TM)$. Consequently, $\bar{\phi}\mathrm{tr}(TM)$ is also a vector subbundle of $S(TM)$. Since $\bar{g}(\bar{\phi}\xi,\bar{\phi}N)=1$, it follows that $\bar{\phi}\mathrm{Rad}\,TM \oplus \bar{\phi}\mathrm{tr}(TM)$ is a non-degenerate vector subbbundle of $S(TM)$, with 2-dimensional fibers. Since $\zeta$ is tangent to $M$, and that $\bar{g}(\bar{\phi}\xi,\zeta)=\bar{g}(\bar{\phi}N,\zeta)=0$, then there exists a non-degenerate distribution $D_{0}$ on $TM$ such that 
\begin{align}\label{e40}
	S(TM)=\{\bar{\phi}\mathrm{Rad}\, TM\oplus \bar{\phi}\mathrm{tr}(TM)\}\perp D_{0}\perp \langle\zeta\rangle_{\mathbb{R}},
\end{align}
where $\langle\zeta\rangle_{\mathbb{R}}$ denotes the line bundle spanned by the structure vector field $\zeta$. Furthermore, it is easy to check that $D_{0}$ is an almost complex distribution with respect to $\bar{\phi}$, that is; $\bar{\phi}D_{0}=D_{0}$. Moreover,  from (\ref{n97}) and (\ref{e40}), $TM$ is decomposed as
\begin{align}\label{n101}
TM=\{\bar{\phi}\mathrm{Rad}\, TM\oplus \bar{\phi}\mathrm{tr}(TM)\}\perp D_{0}\perp \langle\zeta\rangle_{\mathbb{R}}\perp TM^{\perp}.
\end{align}
 If we set $D=\mathrm{Rad}\, TM\perp \bar{\phi}\mathrm{Rad}\, TM\perp D_{0}$ and $D'=\bar{\phi}\mathrm{tr}(TM)$, then (\ref{n101}) becomes
\begin{align}\label{p14}
	TM=D\oplus D'\perp \langle\zeta\rangle_{\mathbb{R}}.
\end{align} 
Here, $D$ is an almost complex distribution and $D'$ is carried by $\bar{\phi}$ into the transversal bundle. Consider the lightlike vector fields $U$ and $V$ given by 
\begin{align}\label{f11}
	U=-\bar{\phi}N, \quad V=-\bar{\phi}\xi, 
\end{align}
together with their corresponding $1$-forms $u$ and $v$ given by
\begin{align}\label{p16}
u(X)=g(X, V), \quad v(X)=g(X, U),
\end{align}
for any $X$ tangent to $M$. Then, from (\ref{p14}), any $X$ tangent to $M$ can be written as $X=P_{1}X+P_{2}X+\eta(X)\zeta$, where $P_{1}$ and $P_{2}$ are projection morphisms of $TM$ onto $D$ and $D'$, respectively. It then follows that, for any $X$ tangent to $M$, we have
\begin{align}\label{p17}
	\bar{\phi}X=\phi X+u(X)N,
\end{align}
where $\phi$ is a (1,1)-tensor field defined on $M$ by $\phi X=\bar{\phi}P_{1}X$. Also, the following relations are easy to verify: 
\begin{align}
	\phi^{2}X&=-X+\eta(X)\zeta+u(X)U,\label{p18}\\
	\phi \zeta&=0,\quad \phi U=0, \quad \phi V=\xi,\label{z8}\\
	\eta(\zeta)=1, \quad &\eta(U)=0, \quad \eta(V)=0, \quad \eta(\phi X)=0,\label{z7}\\
	u(\zeta)=0, \quad &u(U)=1,\quad u(V)=0, \quad u(\phi X)=0,\label{e76}\\
	v(\zeta)=0, \quad v&(U)=0, \quad v(V)=1, \quad v(\phi X)=-\theta(X),\label{z5}\\
 \theta(\zeta)=0, \quad& \theta(U)=0, \quad \theta(V)=0, \quad \theta(\phi X)=v(X),\label{e80}\\
	g(\phi X,\phi Y)=g&(X,Y)-\eta(X)\eta(Y)-u(Y)v(X)-u(X)v(Y),\nonumber\\
	g(\phi X,Y)&=-g(X,\phi Y)-u(X)\theta(Y)-u(Y)\theta(X),\label{p30}
\end{align}
for any $X$ and $Y$ tangent to $M$. The following two lemmas are vital to our study:
\begin{lemma}
Let $(M,g)$ be a lightlike hypersurface of an indefinite Sasakian manifold $(\bar{M}, \bar{\phi}, \zeta, \eta, \bar{g})$, which is tangent to the structure vector field $\zeta$. Then, we have
\begin{align}
\nabla_{X}\zeta&=\phi X,\label{e9}\\
B(X,\zeta)&=u(X),\label{e10}\\
A^{*}_{\xi}\zeta&=V,\label{z6}\\
\nabla^{*}_{X}\zeta&=P\phi X,\label{e11}\\
C(X,\zeta)&=v(X),\label{e12}\\
\nabla_{X}U&=\theta(X)\zeta+\phi A_{N}X+\tau(X)U,\label{e13}\\
B(X,U)&=C(X,V),\label{e14}\\
\nabla_{X}V&=\phi A^{*}_{\xi}X-\tau(X)V,\label{e15}\\
(\nabla_{X}\phi)Y&=\eta(Y)X-g(X,Y)\zeta+u(Y)A_{N}X-B(X,Y)U, \label{e16}\\
(\nabla_{X}u)Y&=-B(X,\phi Y)-\tau(X)u(Y),\label{e17}
\end{align}
for any $X$ and $Y$ tangent to $M$, where $\nabla \phi$ and $\nabla u$ are defined as 
\begin{align}
(\nabla_{X}\phi)Y&=\nabla_{X}\phi Y-\phi \nabla_{X}Y,\label{e19}\\
(\nabla_{X}u)Y&=X\cdot u(Y)-u(\nabla_{X}Y).\nonumber
\end{align}
\end{lemma}
\begin{proof}
Using (\ref{e7}), the last equalilty of (\ref{s10}) and (\ref{p17}), we have 
\begin{align}\label{er3}
\nabla_{X}\zeta+B(X,\zeta)N=\phi X+u(X)N,
\end{align}
for any $X$ tangent to $M$. Comparing components (\ref{er3}), we obtain (\ref{e9}) and (\ref{e10}). Next,  since $S(TM))$ is non-degenerate, the relation $B(X, \zeta)=u(X)$, equivalently; $g(A^{*}_{\xi}\zeta, X)=g(V, X)$, and the fact that $A^{*}_{\xi}$ is screen-valued implies that $A^{*}_{\xi}\zeta=V$, which is relation (\ref{z6}). From (\ref{s12}), (\ref{e80}), (\ref{e9}) and the fact that $X=PX+\theta(X)\xi$, for any $X$ tangent to $M$, we have 
\begin{align}\label{er4}
\nabla^{*}_{X}\zeta+C(X, \zeta)\xi=P\phi X+\theta(\phi X)\xi=P\phi X+v(X)\xi.
\end{align}
Comparing components in (\ref{er4}), we obtain (\ref{e11}) and (\ref{e12}). Next, using (\ref{s10}) and (\ref{f11}), we have 
\begin{align}\label{er5}
\nabla_{X}U&=\bar{\nabla}_{X}U-B(X,U)N\nonumber\\
&=-\bar{\nabla}_{X}\bar{\phi}N-B(X,U)N\nonumber\\
&=-(\bar{\nabla}_{X}\bar{\phi})N-\bar{\phi}\bar{\nabla}_{X}N-B(X,U)N,
\end{align}
for any $X$ tangent to $M$. Then, applying (\ref{e6}), (\ref{s11}), (\ref{er7}), (\ref{er6}), (\ref{p16}) and (\ref{p17}) to (\ref{er5}), we get 
\begin{align}
\nabla_{X}U&=\theta(X)\zeta+\phi A_{N}X+u(A_{N}X)N+\tau(X)U-B(X,U)N\nonumber\\
&=\theta(X)\zeta+\phi A_{N}X+\tau(X)U+\{C(X,V)-B(X,U)\}N.\label{er8}
\end{align}
Comparing components in (\ref{er8}), we obtain (\ref{e13}) and (\ref{e14}). Relation (\ref{e15}) follows in a similar way. Lastly, differentiating (\ref{p17}), and considering (\ref{s10}) and (\ref{s11}), we derive
\begin{align}\label{er11}
\bar{\nabla}_{X}\bar{\phi}Y&=\bar{\nabla}_{X}\phi Y+X\cdot u(Y)N+u(Y)\bar{\nabla}_{X}N\nonumber\\
&=\nabla_{X}\phi Y+B(X,\phi Y)N+X\cdot u(Y)N -u(Y)A_{N}X\nonumber\\
&\quad \quad +u(Y)\tau(X)N,
\end{align}
for any $X$ and $Y$ tangent to $M$. On the other hand, usig (\ref{s10}) (\ref{f11}) and (\ref{p17}), we have 
\begin{align}\label{er10}
\bar{\phi}\bar{\nabla}_{X}Y&=\bar{\phi}\nabla_{X}Y+B(X,Y)\bar{\phi}N\nonumber\\
&=\phi \nabla_{X}Y+u(\nabla_{X}Y)N-B(X,Y)U,
\end{align}
for any $X$ and $Y$ tangent to $M$.
Subtracting (\ref{er10}) from (\ref{er11}), and considerinig (\ref{e6}), we get 
\begin{align}\label{er11}
\eta(Y)X-g(X,Y)\zeta&=(\bar{\nabla}_{X}\bar{\phi})Y=\bar{\nabla}_{X}\bar{\phi}Y-\bar{\phi}\bar{\nabla}_{X}Y\nonumber\\
&=(\nabla_{X}\phi)Y-u(Y)A_{N}X+B(X,Y)U\nonumber\\
&\quad \quad+\{(\nabla_{X}u)Y+B(X,\phi Y)+u(Y)\tau(X)\}N.
\end{align}
Comparing components in (\ref{er11}), we get (\ref{e16}) and (\ref{e17}).
\end{proof}

\begin{lemma}\label{z4}
Let $(M,g)$ be a lightlike hypersurface of an indefinite Sasakian manifold $(\bar{M}, \bar{\phi}, \zeta, \eta, \bar{g})$, which is tangent to the structure vector field $\zeta$. Then, we have
\begin{align}
R(X,Y)\zeta&=\eta(Y)X+u(Y)A_{N}X-\eta(X)Y-u(X)A_{N}Y,\label{e21}\\
R(X,Y)U&=(\nabla_{X}\phi A_{N})Y-(\nabla_{Y}\phi A_{N})X+\theta(Y)\phi X-\theta(X)\phi Y\nonumber\\
&\quad +\tau(Y)\phi A_{N}X-\tau(X)\phi A_{N}Y +2d\tau(X,Y)U\nonumber\\
&\quad +\{2d\theta(X,Y)+\theta(X)\tau(Y)-\theta(Y)\tau(X)\}\zeta\label{e72}\\
R(X,Y)V&=(\nabla_{X}\phi A^{*}_{\xi})Y-(\nabla_{Y}\phi A^{*}_{\xi})X\nonumber\\
&\quad+\tau(X)\phi A^{*}_{\xi}Y-\tau(Y)\phi A^{*}_{\xi}X -2d\tau(X,Y)V,\label{e73}
\end{align}
for any $X$ and $Y$ tangent to $M$, where $\nabla \phi A_{N}$ and $\nabla \phi A^{*}_{\xi}$ are defined as
\begin{align}
(\nabla_{X}\phi A_{N})Y&=\nabla_{X}\phi A_{N}Y-\phi A_{N}\nabla_{X}Y,\label{e74}\\
(\nabla_{X}\phi A^{*}_{\xi})Y&=\nabla_{X}\phi A^{*}_{\xi}Y-\phi A^{*}_{\xi}\nabla_{X}Y.\label{e75}
\end{align}
\end{lemma}
\begin{proof}
By direct calculations while using (\ref{e9}), (\ref{e16}) and (\ref{e19}), we derive
\begin{align}
R(X,Y)\zeta&=\nabla_{X}\nabla_{Y}\zeta-\nabla_{Y}\nabla_{X}\zeta-\nabla_{[X,Y]}\zeta\nonumber\\
&=\nabla_{X}\phi Y-\nabla_{Y}\phi X-\phi [X,Y]\nonumber\\
&=(\nabla_{X}\phi)Y-(\nabla_{Y}\phi)X\nonumber\\
&=\eta(Y)X+u(Y)A_{N}X-\eta(X)Y-u(X)A_{N}Y,\nonumber
\end{align}
for any $X$ and $Y$ tangent to $M$, which proves (\ref{e21}). Next, using (\ref{e13}), we derive 
\begin{align}\label{e70}
\nabla_{X}\nabla_{Y}U&=X\cdot \theta(Y)\zeta+\theta(Y)\phi X+\nabla_{X}\phi A_{N}Y+X\cdot \tau(Y)U\nonumber\\
&\quad +\theta(X)\tau(Y)\zeta+\tau(Y)\phi A_{N}X+\tau(X)\tau(Y)U,
\end{align}
for any $X$ and $Y$ tangent to $M$. Also, by (\ref{e13}), we have 
\begin{align}\label{e71}
\nabla_{[X,Y]}U=\theta([X,Y])\zeta+\phi A_{N}[X,Y]+\tau([X,Y])U.
\end{align}
Replacing (\ref{e70}) and (\ref{e71}) in 
\begin{align*}
R(X,Y)U=\nabla_{X}\nabla_{Y}U-\nabla_{Y}\nabla_{X}U-\nabla_{[X,Y]}U,
\end{align*}
we get (\ref{e72}). Relation (\ref{e73}) can be proved in the same way.
\end{proof}

\begin{corollary}
For any $X$ and $Y$ tangent to $M$, we have
\begin{align}
g((\nabla_{X}\phi A_{N})Y, V)&=-g(\phi A^{*}_{\xi}X, \phi A_{N}Y)+B(X,V)C(Y,U),\label{e81}\\
g((\nabla_{X}\phi A^{*}_{\xi}X)Y, V)&=-g(\phi A^{*}_{\xi}X, \phi A^{*}_{\xi}Y)+B(X,V)B(Y,U).\label{e82}
\end{align}
\end{corollary}
\begin{proof}
Using (\ref{s14}), (\ref{e76}), (\ref{e80}),  (\ref{e15}) and (\ref{e74}), we have 
\begin{align*}
g((\nabla_{X}\phi A_{N})Y, V)&=g(\nabla_{X}\phi A_{N}Y, V)-g(\phi A_{N}\nabla_{X}Y, V)\nonumber\\
&=-g(\phi A_{N}Y, \nabla_{X}V)+(\nabla_{X}g)(\phi A_{N}Y, V)\nonumber\\
&=-g(\phi A_{N}Y, \phi A^{*}_{\xi}X)+B(X,V)\theta(\phi A_{N}Y)\nonumber\\
&=-g(\phi A_{N}Y, \phi A^{*}_{\xi}X)+B(X,V)C(Y,U),
\end{align*}
for any $X$ and $Y$ tangent to $M$, which proves (\ref{e81}). Finally, (\ref{e82}) also follows easily as in (\ref{e81}), which completes the proof.
\end{proof}
\section{Locally symmetric lightlike hypersurfaces}\label{loc}
In this section, we prove that a lightlike hypersurface $(M,g)$ of an indefinite Sasakian manifold $(\bar{M}, \bar{\phi}, \zeta, \eta, \bar{g})$, which is tangent to structure vector field $\zeta$ is not locally symmetric (see Theorems \ref{c8} and \ref{e47}). Let $R$ be the curvature tensor of the induced linear connection $\nabla$. For every $X$, $Y$, $Z$ and $W$ tangent to $M$, the covariant derivative $\nabla R$ of $R$ is given by
\begin{align}\label{e130}
(\nabla_{X}R)(Y,Z)&W=\nabla_{X}R(Y,Z)W\nonumber\\
&-R(\nabla_{X}Y,Z)W-R(Y, \nabla_{X}Z)W-R(Y,Z)\nabla_{X}W.
\end{align}
Then, $M$ is said to be  {\it locally symmetric} \cite{Sahin, Oscar1, Oscar2, Sahins} whenever $\nabla R=0$. Next, our results in this section make use of the following important lemmas:

\begin{lemma}
Let $(M,g)$ be a lightlike hypersurface of an indefinite Sasakian manifold $(\bar{M}, \bar{\phi}, \zeta, \eta, \bar{g})$, which is tangent to the structure vector field $\zeta$. Then, $\nabla R$ satisfies the relation below
\begin{align}
(\nabla_{X}R)(Y,Z)\zeta&=\{u(X)\theta(Z)+g(\phi X,Z)\}Y+u(Z)(\nabla_{X}A_{N})Y\nonumber\\
&\quad -\{B(X,\phi Z)+\tau(X)u(Z)\}A_{N}Y\nonumber\\
&\quad -\{u(X)\theta(Y)+g(\phi X,Y)\}Z-u(Y)(\nabla_{X}A_{N})Z\nonumber\\
&\quad+\{B(X,\phi Y)+\tau(X)u(Y)\}A_{N}Z-R(Y,Z)\phi X,\label{e25}
\end{align}
for any $X$, $Y$ and $Z$ tangent to $M$, where $\nabla A_{N}$ is given by 
\begin{align}\label{e46}
(\nabla_{X}A_{N})Y=\nabla_{X}A_{N}Y-A_{N}\nabla_{X}Y,
\end{align}
for any $X$ and $Y$ tangent to $M$.
\end{lemma}
\begin{proof}
Using the definition of $\nabla R$ in (\ref{e130}), relations (\ref{e9}) and (\ref{e21}), we derive 
\begin{align}
(\nabla_{X}&R)(Y,Z)\zeta\nonumber\\
&=\nabla_{X}R(Y,Z)\zeta-R(\nabla_{X}Y, Z)\zeta-R(Y, \nabla_{X}Z)\zeta-R(X,Y)\nabla_{X}\zeta,\nonumber\\
&=\{(\nabla_{X}\eta)Z\}Y+\{(\nabla_{X}u)Z\}A_{N}Y+u(Z)(\nabla_{X}A_{N})Y-\{(\nabla_{X}\eta)Y\}Z\nonumber\\
&\quad \quad -\{(\nabla_{X}u)Y\}A_{N}Z-u(Y)(\nabla_{X}A_{N})Z-R(Y,Z)\phi X,\label{e36}
\end{align}
for any $X$, $Y$ and $Z$ tangent to $M$. On the other hand, a direct calculation shows 
\begin{align}
(\nabla_{X}\eta)Y&=X\cdot \eta(Y)-\eta(\nabla_{X}Y)\nonumber\\
&=u(X)\theta(Y)+g(\phi X,Y),\label{e35}
\end{align}
for any $X$ and $Y$ tangent to $M$. Replacing (\ref{e17}) and (\ref{e35}) in (\ref{e36}), we obtain (\ref{e25}), which ends the proof.
\end{proof}

\begin{lemma}\label{lemma2}
For every $X$ tangent to $M$, the following holds
\begin{align}\label{vl1}
g((\nabla_{X}A_{N})\zeta, \zeta)=-C(\zeta, P\phi X)+\theta(X).
\end{align}
\end{lemma}
\begin{proof}
Using (\ref{er7}), (\ref{s14}), (\ref{z5}) , (\ref{e9}) and (\ref{e12}), we derive 
\begin{align}
g((\nabla_{X}A_{N})\zeta, \zeta)&=g(\nabla_{X}A_{N}\zeta, \zeta)-g(A_{N}\nabla_{X}\zeta, \zeta)\nonumber\\
&=X\cdot v(\zeta)-g(A_{N}\zeta, \nabla_{X}\zeta)-(\nabla_{X}g)(A_{N}\zeta, \zeta)-g(A_{N}\nabla_{X}\zeta, \zeta)\nonumber\\
&=-g(A_{N}\zeta, \phi X)-g(A_{N}\phi X, \zeta)\nonumber\\
&=-C(\zeta, P\phi X)-C(\phi X, \zeta),\nonumber\\
&=-C(\zeta, P\phi X)-v(\phi X)\nonumber\\
&=-C(\zeta, P\phi X)+\theta(X),\nonumber
\end{align}
for any $X$ tangent to $M$. 
\end{proof}

\noindent Now, we prove the following result:
\begin{theorem}\label{c8}
There exist no lightlike hypersurface $(M,g)$ of an indefinite Sasakian manifold $(\bar{M}, \bar{\phi}, \zeta, \eta, \bar{g})$, which is tangent to the structure vector field $\zeta$ and  locally symmetric.
\end{theorem}

\begin{proof}
Assume $M$ is tangent to $\zeta$ and locally symmetric, i.e. $\nabla R=0$. Then relation (\ref{e25}) reduces to 
\begin{align}
&\{u(X)\theta(Z)+g(\phi X,Z)\}Y+u(Z)(\nabla_{X}A_{N})Y-\{B(X,\phi Z)\nonumber\\
&\quad +\tau(X)u(Z)\}A_{N}Y-\{u(X)\theta(Y)+g(\phi X,Y)\}Z-u(Y)(\nabla_{X}A_{N})Z\nonumber\\
&\quad +\{B(X,\phi Y)+\tau(X)u(Y)\}A_{N}Z-R(Y,Z)\phi X=0,\label{e26}
\end{align}
for any $X$, $Y$ and $Z$ tangent to $M$. Taking $X=U$, $Z=\zeta$ in  (\ref{e26}), and considering (\ref{z8}), we get 
\begin{align}
\theta(Y)\zeta&+u(Y)(\nabla_{U}A_{N})\zeta-\{B(U, \phi Y)+\tau(U)u(Y)\}A_{N}\zeta=0.\label{e27}
\end{align}
Taking the inner product of (\ref{e27}) with $\zeta$, and considering (\ref{e12}) and (\ref{vl1}), we get 
\begin{align*}
\theta(Y)-u(Y)C(\zeta, P\phi U)=0,
\end{align*}
which reduces to $\theta(Y)=0$, for any $Y$ tangent to $M$. This contradicts (\ref{er2}).
\end{proof}

 A lightlike submanifold $(M,g)$ of a semi-Riemannian manifold $(\bar{M}, \bar{g})$ is said to be locally recurrent if its curvature tensor $R$ satisfies the relation
\begin{align}\label{e45}
(\nabla_{X}R)(Y,Z)W=\rho(X)R(Y,Z)W,
\end{align}
for any $X$, $Y$, $Z$ and $W$ tangent to $M$, where $\rho$ is a smooth 1-form on $M$. Furthermore, it is easy to see that $M$ is locally symmtric whenever $\rho=0$.

\begin{theorem}\label{e47}
There exist no lightlike hypersurface $(M,g)$ of an indefinite Sasakian manifold $(\bar{M}, \bar{\phi}, \zeta, \eta, \bar{g})$, which is tangent to the structure vector field $\zeta$ and  locally recurrent.
\end{theorem}
\begin{proof}
Using (\ref{e21}), (\ref{e25}) and (\ref{e45}), we have 
\begin{align}
\{u&(X)\theta(Z)+g(\phi X,Z)\}Y+u(Z)(\nabla_{X}A_{N})Y-\{B(X,\phi Z)\nonumber\\
&+\tau(X)u(Z)\}A_{N}Y-\{u(X)\theta(Y)+g(\phi X,Y)\}Z-u(Y)(\nabla_{X}A_{N})Z\nonumber\\
&+\{B(X,\phi Y)+\tau(X)u(Y)\}A_{N}Z-R(Y,Z)\phi X\nonumber\\
&=\rho(X)\{\eta(Z)Y+u(Z)A_{N}Y-\eta(Y)Z-u(Y)A_{N}Z\},\label{e50}
\end{align}
for any $X$, $Y$ and $Z$ tangent to $M$. With $X	=U$ and $Z=\zeta$ in (\ref{e50}), we get 
\begin{align}
-\theta(Y)\zeta&-u(Y)(\nabla_{U}A_{N})\zeta+\{B(U, \phi Y)+\tau(U)u(Y)\}A_{N}\zeta\nonumber\\
&=\rho(U)\{Y-\eta(Y)\zeta-u(Y)A_{N}\zeta\},\label{vl2}
\end{align}
for any $Y$ tangent to $M$. Taking the inner product of (\ref{vl2}) with $\zeta$, and  considering (\ref{e12}) and (\ref{vl1}), we get $\theta=0$, which contradicts (\ref{er2}).
\end{proof}

\section{Semi-symmetric lightlike hypersurfaces}\label{semi-sym}
 Let $(M,g)$ be a lightlike hypersurface of a semi-Riemannian manifold $(\bar{M}, \bar{g})$. Then, $M$ is said to be {\it semi-symmetric} \cite{Sahins} if its curvature tensor $R$ satisfies the condition $R\cdot R=0$, that is; for any $X, Y, X_{1}, X_{2}, X_{3}$ and $X_{4}$ tangent to $M$,  
\begin{align}
(R(X,Y)\cdot R)&(X_{1}, X_{2}, X_{3}, X_{4})=-R(R(X,Y)X_{1}, X_{2}, X_{3}, X_{4})\nonumber\\
&-R(X_{1}, R(X,Y)X_{2}, X_{3}, X_{4})-R(X_{1}, X_{2}, R(X,Y)X_{3}, X_{4})\nonumber\\
&\quad \quad -R(X_{1}, X_{2}, X_{3}, R(X,Y)X_{4})=0.\label{z1}
\end{align}
Our next result will be to show that a lightlike hypersurface of an indefinite Sasakian manifold $(\bar{M}, \bar{\phi}, \zeta, \eta, \bar{g})$, which is tangent to the structure vector field $\zeta$ is never semi-symmetric. In that line, we will need the following important lemma. 
\begin{lemma}
The following curvature relatiuons holds
\begin{align}
R(U,\zeta)\zeta&=U-A_{N}\zeta,\quad R(V, \zeta)\zeta=V,\label{z2}\\
g(R(U,V)\zeta, &\zeta)=-1,\quad g(R(U,\zeta)V,\zeta)=-1. \label{z3}
\end{align}
\end{lemma}
\begin{proof}
The two relations in (\ref{z2}) follows directly from (\ref{e21}) of Lemma \ref{z4}. Again, by (\ref{e21}, together with (\ref{e76}), (\ref{z5}), and (\ref{e12}), we have 
\begin{align*}
g(R(U,V)\zeta, \zeta)=-u(U)g(A_{N}V, \zeta)=-C(V, \zeta)=-v(V)=-1,
\end{align*}
which proves the first relation of (\ref{z3}). Futhermore, using (\ref{s13}), (\ref{s14}), (\ref{z8}), (\ref{z7}), (\ref{e80}), (\ref{e10}), (\ref{z6}), (\ref{e73}) and (\ref{e82}), we derive
\begin{align*}
g(R(U,\zeta)V,\zeta)&=g((\nabla_{U}\phi A^{*}_{\xi})\zeta, \zeta)-g((\nabla_{\zeta}\phi A^{*}_{\xi})U, \zeta)\\
&=g(\nabla_{U}\phi A^{*}_{\xi}\zeta, \zeta)-g(\nabla_{\zeta}\phi A^{*}_{\xi}U,\zeta)\\
&=g(\nabla_{U}\phi V, \zeta)+g(\phi A^{*}_{\xi}U,\nabla_{\zeta}\zeta)+(\nabla_{\zeta}g)(\phi A^{*}_{\xi}U, \zeta)\\
&=g(\nabla_{U}\xi, \zeta)+B(\zeta, \zeta)\theta(\phi A^{*}_{\xi}U)\\
&=-g(A^{*}_{\xi}U, \zeta)+u(\zeta)B(U,U)\\
&=-B(U,\zeta)\\
&=-u(U)\\
&=-1,
\end{align*}
which proves the second relation in (\ref{z3}).
\end{proof}

\noindent Next, we prove the following result:

\begin{theorem}\label{e471}
	There exist no lightlike hypersurface $(M,g)$ of an indefinite Sasakian manifold $(\bar{M}, \bar{\phi}, \zeta, \eta, \bar{g})$, which is tangent to the structure vector field $\zeta$ and  semi-symmetric. 
\end{theorem}
\begin{proof}
Suppose that $M$ is tangent to $\zeta$ and semi-symmetric. Then, taking $X=V$, $X_{1}=U$ and $Y=X_{2}=X_{3}=X_{4}=\zeta$ in (\ref{z1}), we obtain 
\begin{align}\label{z10}
R(R(V,\zeta)U,\zeta,\zeta,\zeta)&+R(U,R(V,\zeta)\zeta,\zeta,\zeta)\nonumber\\
&+R(U,\zeta,R(V,\zeta)\zeta,\zeta)+R(U,\zeta,\zeta,R(V,\zeta)\zeta)=0.
\end{align}
But, using (\ref{e12}) and (\ref{e21}), the first term of (\ref{z10}) simplies as 
\begin{align}\label{z11}
R(R(V,\zeta)U,\zeta,\zeta,\zeta)&=g(R(R(V,\zeta)U, \zeta)\zeta, \zeta)\nonumber\\
&=g(R(V,\zeta)U-\eta(R(V,\zeta)U)\zeta, \zeta)=0.
\end{align}
Then applying (\ref{z2}) and (\ref{z11}) to (\ref{z10}), we get 
\begin{align}\label{z12}
g(R(U,V)\zeta,\zeta)+g(R(U,\zeta)V,\zeta)+g(R(U,\zeta)\zeta, V)=0.
\end{align}
The last term of (\ref{z12}) vanishes according to (\ref{e10}), (\ref{e14}) and the first relation in (\ref{z2}); that is
\begin{align*}
g(R(U,\zeta)\zeta, V)=g(U-A_{N}\zeta, V)=g(U,V)-C(\zeta, V)=1-B(U,\zeta)=0.
\end{align*}
Thus, (\ref{z12}) reduces to 
\begin{align*}
g(R(U,V)\zeta,\zeta)+g(R(U,\zeta)V,\zeta)=0,
\end{align*}
which leads to the obvious contradiction  $-1-1=0$ when the two relations in (\ref{z3}) are substituted into it, which proves our result.
\end{proof}

\section{Semi-parallel lightlike hypersurfaces}\label{semi-par}
 Let $(M,g)$ be a lightlike submanifold of a semi-Riemannian manifold $(\bar{M}, \bar{g})$.  Let $\tilde{R}$ denote the curvature tensor of the Van der Waerden-Bortolotti connection $\tilde{\nabla}:=\nabla\oplus \nabla^{t}$, where $\nabla$ and $\nabla^{t}$ denotes the induced connections in $TM$ and the transversal bundle $\mathrm{tr}(TM)$ by $\bar{\nabla}$, respectively. Then, it is easy to see that 
\begin{align}\label{e90}
(\tilde{R}(X,Y)\cdot h)(Z,W)&=R^{t}(X,Y)h(Z,W)\nonumber\\
&\quad -h(R(X,Y)Z,W)-h(Z, R(X,Y)W),
\end{align}
for any $X$, $Y$, $Z$ and $W$ tangent to $M$. Here, $h$ denotes the second fundamental form of $M$ and $R^{t}$ the transversal curvature tensor of $M$, defined as 
\begin{align}\label{e91}
R^{t}(X,Y)V&=\nabla^{t}_{X}\nabla^{t}_{Y}V-\nabla^{t}_{Y}\nabla^{t}_{X}V-\nabla^{t}_{[X,Y]}V,
\end{align} 
for any transversal vector field $V$. The transversal connection $\nabla^{t}$ is said to be trivial if $R^{t}=0$. Similar to John Deprez \cite{John}, we call $M$ {\it semi-parallel} if 
\begin{align}\label{e93}
\tilde{R}(X,Y)\cdot h=0,
\end{align}
for any $X$ and $Y$ tangent to $M$. In case $M$ is a lightlike hypersurface of $\bar{M}$, we have the following:
\begin{lemma}\label{lemma11}
A lightlike hypersurface $(M,g)$ of a semi-Riemannian manifold $(\bar{M}, \bar{g})$ is semi-parallel if and only if 
\begin{align}\label{e89}
2d\tau(X,Y)B(Z,W)-B(R(X,Y)Z,W)-B(Z, R(X,Y)W)=0,
\end{align}
for any $X$, $Y$, $Z$ and $W$ tangent to $M$, where  
\begin{align*}
2d\tau(X,Y)=X\cdot \tau(Y)-Y\cdot \tau(X)-\tau([X,Y]).
\end{align*}
\end{lemma}
\begin{proof}
Using the equation $\nabla^{t}_{X}N=\tau(X)N$ together with (\ref{e91}), we derive 
\begin{align}\label{e134}
R^{t}(X,Y)N=2d\tau(X,Y)N,
\end{align}
for any $X$ and $Y$ tangent to $M$. Applying (\ref{e134}) to (\ref{e90}) and (\ref{e93}), we obtain our lemma.   
\end{proof}

\noindent Next, using Lemma \ref{lemma11}, we prove the following result:
\begin{theorem}\label{e4711}
	There exist no lightlike hypersurface $(M,g)$ of an indefinite Sasakian manifold $(\bar{M}, \bar{\phi}, \zeta, \eta, \bar{g})$, which is tangent to the structure vector field $\zeta$ and  semi-parallel.
\end{theorem}
\begin{proof}
Suppose that $M$ is tangent to $\zeta$ and semi-parallel. Then, taking $Z=W=\zeta$ in (\ref{e89}) of Lemma \ref{lemma11}, and noting that $B$ is symmetric, we get 
\begin{align}\label{e100}
2d\tau(X,Y)B(\zeta,\zeta)-2B(R(X,Y)\zeta,\zeta)=0,
\end{align}
for any $X$ and $Y$ tangent to $M$. In view of (\ref{e10}), we have  
\begin{align}
B(\zeta, \zeta)&=u(\zeta)=0,\label{102}
\end{align}
\begin{align}
B(R(X,Y)\zeta,\zeta)&=u(R(X,Y)\zeta)=g(R(X,Y)\zeta, V).\label{e103}
\end{align}
Thus, replacing (\ref{102}) and (\ref{e103}) in (\ref{e100}), we get 
\begin{align}\label{e101}
g(R(X,Y)\zeta, V)=0.
\end{align}
Applying the curvature condition (\ref{e21}) to (\ref{e101}) and considering (\ref{e14}), we get 
\begin{align}\label{e104}
u(X)\eta(Y)+B(X,U)u(Y)-\eta(X)u(Y)-u(X)B(Y,U)=0.
\end{align}
Taking $Y=U$ in (\ref{e104}) and using (\ref{e76}), we get 
\begin{align}\label{e106}
B(X,U)-\eta(X)-u(X)B(U,U)=0.
\end{align}
Since $S(TM)$ is non-degenerate, (\ref{e106}) reduces to 
\begin{align}\label{e107}
A^{*}_{\xi}U=\zeta+B(U,U)V.
\end{align}
The inner product of (\ref{e107}) with $V$ leads to 
\begin{align}\label{e108}
B(U,V)=0.
\end{align}
Next, taking $Z=V$ and $W=\zeta$ in (\ref{e89}) of Lemma \ref{lemma11}, we get 
\begin{align}\label{e109}
2d\tau(X,Y)B(V,\zeta)-B(R(X,Y)V,\zeta)-B(V, R(X,Y)\zeta)=0.
\end{align}
But, in view of (\ref{e76}),  (\ref{e10}), (\ref{e21}), (\ref{e73}) and (\ref{e82}), we derive 
\begin{align}
B(V,\zeta)=u(V)=0,\label{e111}
\end{align}
\begin{align}
B(R(X,Y)V,\zeta)&=u(R(X,Y)V)\nonumber\\
&=g(R(X,Y)V, V)\nonumber\\
&=B(X,V)B(Y,U)-B(Y,V)B(X,U),\label{e112}
\end{align}
\begin{align}
B(V, R(X,Y)\zeta)&=\eta(Y)B(X,V)+u(Y)B(A_{N}X,V)-\eta(X)B(Y,V)\nonumber\\
&\quad -u(X)B(A_{N}Y,V).\label{e113}
\end{align}
Replacing (\ref{e111}), (\ref{e112}) and (\ref{e113}) in (\ref{e109}), we get
\begin{align}\label{e114}
B(X,&V)B(Y,U)-B(Y,V)B(X,U)+\eta(Y)B(X,V)\nonumber\\&+u(Y)B(A_{N}X,V)-\eta(X)B(Y,V)-u(X)B(A_{N}Y,V)=0.
\end{align}
Taking $Y=\zeta$ in (\ref{e114}) and considering (\ref{e10}), we get 
\begin{align}\label{e115}
2B(X,V)+u(X)B(A_{N}\zeta, V)=0.
\end{align}
With $X=U$ in (\ref{e115}) and considering (\ref{e108}), we $B(A_{N}\zeta, V)=0$. Therefore,  
\begin{align}\label{e116}
B(X,V)=0,\;\; \mbox{i.e.}\;\; A^{*}_{\xi}V=0.
\end{align}
Moving forward, we consider $Z=U$ and $W=\zeta$ in (\ref{e89}) to obtain
\begin{align}\label{e117}
2d\tau(X,Y)B(U,\zeta)-B(R(X,Y)U,\zeta)-B(U, R(X,Y)\zeta)=0.
\end{align}
Using (\ref{e76}), (\ref{e10}), (\ref{e21}), (\ref{e72}), (\ref{e81}) and (\ref{e116}), we derive
\begin{align}
B(U,\zeta)=u(U)=1,\label{e118}
\end{align}
\begin{align}
B(R(X,Y)&U,\zeta)=u(R(X,Y)U)\nonumber\\
&=g(R(X,Y)U, V)\nonumber\\
&=-g(\phi A^{*}_{\xi}X, \phi A_{N}Y)+g(\phi A^{*}_{\xi}Y, \phi A_{N}X)+2d\tau(X,Y),\label{e119}
\end{align}
\begin{align}
B(U, R(X,Y)\zeta)&=\eta(Y)B(X,U)+u(Y)B(A_{N}X,U)-\eta(X)B(Y,U)\nonumber\\
&\quad-u(X)B(A_{N}Y,U).\label{e120}
\end{align}
Replacing (\ref{e118}), (\ref{e119}) and (\ref{e120}) in (\ref{e117}), we get 
\begin{align}
g(\phi &A^{*}_{\xi}X,  \phi A_{N}Y)-g(\phi A^{*}_{\xi}Y, \phi A_{N}X)-\eta(Y)B(X,U)\nonumber\\
&-u(Y)B(A_{N}X,U)+\eta(X)B(Y,U)+u(X)B(A_{N}Y,U)=0.\label{e121}
\end{align}
Taking $Y=V$ in (\ref{e121}) and considering (\ref{e116}), we get 
\begin{align}\label{e122}
g(\phi &A^{*}_{\xi}X,  \phi A_{N}V)+u(X)B(A_{N}V,U)=0.
\end{align}
With $X=U$ in (\ref{e122}), and considering (\ref{e107}), we get 
\begin{align}\label{e135}
B(A_{N}V,U)=0.
\end{align}
But, using (\ref{e12}), (\ref{e14}), (\ref{e107}) and (\ref{e116}), we derive
\begin{align}\label{e136}
B(A_{N}V,U)&=g(A_{N}V, A^{*}_{\xi}U)=B(U,U)B(U, V)+C(V,\zeta)=1.
\end{align}
Comparing (\ref{e135}) and (\ref{e136}), we get a  contradiction.
\end{proof}

\subsection*{Concluding remarks}

Although lightlike hypersurfaces of an indefinite Sasakian manifold $(\bar{M}, \bar{\phi}, \zeta, \eta, \bar{g})$, which are tangent to the structure vector field $\zeta$ have been studed under the assumption of them being locally symmetric, semi-symmetric and semi-parallel in  \cite{Oscar1} and \cite{Gupta}, the results of this paper shows clearly that such hypersufaces are non-existent. Some other types of non-existent lightlike hypersurfaces of an indefinite Sasakian manifold can be found in \cite{Jin2} and in the recent work of \cite{Ssekajja}.

\end{document}